\documentclass[a4paper]{article}
\usepackage{amsmath}
\usepackage{amssymb}

\newtheorem{theorem}{Theorem}

\newtheorem{corollary}[theorem]{Corollary}

\newtheorem{lemma}[theorem]{Lemma}

\newtheorem{proposition}[theorem]{Proposition}

\newenvironment{proof}[1][Proof]{\noindent\textbf{#1.} }{\ \rule{0.5em}{0.5em}}
\input{tcilatex}
\begin{document}

\begin{center}
\textbf{Certain Metric Properties of Level Hypersurfaces}

\medskip 

Pisheng Ding
\end{center}

\begin{quote}
\textsc{Abstract. }{\small This note establishes several integral identities
relating certain metric properties of level hypersurfaces of Morse functions.%
}
\end{quote}

\section{Introduction}

Let $f$ be a $C^{2}$ Morse function on an open connected subset $\Omega $ of 
$%
\mathbb{R}
^{n+1}$ where $n\geq 2$. Suppose that $a$ and $b$ are values of $f$ such\
that $f^{-1}([a,b])$ is compact. For $t\in \lbrack a,b]$, let $\nu (t)$ be
the ($n$-dimensional) volume of the level-$t$ set $f^{-1}(t)$. Note that,
since $f$ is a Morse function, $\nu (t)$ is well-defined even if $t$ is a
critical value and that $\nu :[a,b]\rightarrow 
\mathbb{R}
$ is continuous. At each regular point (i.e., noncritical point) on $%
f^{-1}(t)$, let $\mathbf{N=}-\nabla f/\left\vert \nabla f\right\vert $. This
choice of unit normal induces a Gauss map $G$ on the set of regular points
on $f^{-1}(t)$, with $G(p)=\mathbf{N}(p)\in S^{n}$. The mean curvature $H$
and the Gaussian curvature $K$ are defined on the set of regular points on $%
f^{-1}(t)$ by the standard definitions%
\begin{equation*}
H=\frac{1}{n}\limfunc{Tr}dG\quad \text{and}\quad K=\det dG\text{\thinspace .}
\end{equation*}%
We henceforth view $H$ and $K$ as functions on the set of regular points of $%
f^{-1}([a,b])$; i.e., $H(p)$ and $K(p)$ are the mean curvature and Gaussian
curvature of $f^{-1}(f(p))$ at $p$.

We now state our main results, in which $d\mu $ is the Lebesgue measure on $%
\mathbb{R}
^{n+1}$ and $\partial _{i}$ denotes the $i$-th partial derivative.

\medskip

\noindent \textbf{Theorem}\quad \textit{Assume the preceding assumptions and
notation.}

\smallskip

\noindent (a)\quad $\nu (b)-\nu (a)=n\int_{f^{-1}([a,b])}H\,d\mu \,$.

\smallskip

\noindent (b)$\quad \int_{a}^{b}\nu (t)\,dt=\int_{f^{-1}([a,b])}\left\vert
\nabla f\right\vert d\mu $\thinspace .

\smallskip

\noindent (c)\quad $\int_{f^{-1}([a,b])}K\partial _{i}f\,d\mu =0$\textit{\
for each }$i\in \{1,\cdots ,n+1\}$.

\medskip

(Implicit in these results is the assertion that the functions $H$ and $%
K\partial _{i}f$ are integrable on $f^{-1}([a,b])$. This is a consequence of 
$f$ being a Morse function, as we shall demonstrate.)

\section{Two Preparatory Results}

In many results, we assume the following hypothesis.

\begin{quote}
\textbf{Hypothesis \dag }: $f$ is a $C^{2}$ Morse function on an open
connected subset $\Omega $ of $%
\mathbb{R}
^{n+1}$ where $n+1\geq 3$; $a$ and $b$ are values of $f$ such\ that $%
f^{-1}([a,b])$ is compact.
\end{quote}

\begin{lemma}
\label{Lemma Key}Assume Hypothesis \dag . Suppose that $g$ is a function
that is continuous on the set of regular points in $f^{-1}([a,b])$ and
integrable on $f^{-1}([a,b])$. Then%
\begin{equation}
\int_{f^{-1}([a,b])}g\,d\mu =\int_{a}^{b}\left( \int_{f^{-1}(t)}\frac{g}{%
\left\vert \nabla f\right\vert }d\sigma \right) dt\,\text{,}  \notag
\end{equation}%
where $d\sigma $ is the ($n$-dimensional) volume form on $f^{-1}(t)$ and $%
\int_{f^{-1}(t)}\left( g/\left\vert \nabla f\right\vert \right) d\sigma $ is
only defined for $t$ a regular value.\footnote{%
The \textquotedblleft outer\textquotedblright\ integral $\int_{a}^{b}\cdots
dt$ on the right may first be interpreted as an improper Riemann integral.
Once the formula is proven, applying it to $\left\vert g\right\vert $ shows
that the one-variable function $\varphi (t):=\int_{f^{-1}(t)}\left(
g/\left\vert \nabla f\right\vert \right) d\sigma $ is absolutely integrable
over $[a,b]$, since $\left\vert \varphi (t)\right\vert \leq
h(t):=\int_{f^{-1}(t)}\left( \left\vert g\right\vert /\left\vert \nabla
f\right\vert \right) d\sigma $ and $\int_{a}^{b}h(t)dt=\int_{f^{-1}([a,b])}%
\left\vert g\right\vert d\mu $. Hence, $\int_{a}^{b}\varphi (t)dt$ may also
be interpreted as a Lebesgue integral.}
\end{lemma}

\begin{proof}
There are two cases, according as whether $[a,b]$ contains a critical value.

Case 1: $[a,b]$ is free of critical values. For each $p\in f^{-1}(a)$, let $%
t\mapsto F(p,t)$ be the integral curve for the field $\nabla f/\left\vert
\nabla f\right\vert ^{2}$.\ The map $F:f^{-1}(a)\times \lbrack
a,b]\rightarrow f^{-1}([a,b])$ is then a diffeomorphism, providing the
transformation of variables that results in the claimed formula. (In detail,
take a coordinate patch $U$ on $f^{-1}(a)$ and apply Fubini's theorem to $%
U\times \lbrack a,b]\overset{F|_{U\times \lbrack a,b]}}{\longrightarrow }%
F(U\times \lbrack a,b])$.)

Case 2: $[a,b]$ contains a critical value. Let $S$ be the (finite) set of
critical values in $[a,b]$. Then, $(a,b)\smallsetminus S$ is a disjoint
union of finitely many intervals $I_{j}:=(c_{j},c_{j+1})$ of regular values.
As $f^{-1}([a,b])=\cup _{j}f^{-1}(I_{j})\cup f^{-1}\left( S\cup
\{a,b\}\right) $ and $f^{-1}\left( S\cup \{a,b\}\right) $ has Lebesgue
measure zero (as a subset of $%
\mathbb{R}
^{n+1}$),%
\begin{equation*}
\int_{f^{-1}([a,b])}g\,d\mu =\sum_{j}\int_{f^{-1}(I_{j})}g\,d\mu \,\text{.}
\end{equation*}%
Applying Case 1 to $f^{-1}([c_{j}+\epsilon ,c_{j+1}-\delta ])$ and letting $%
\epsilon ,\delta \rightarrow 0^{+}$, we have%
\begin{eqnarray*}
\int_{f^{-1}(I_{j})}g\,d\mu &=&\lim_{\epsilon ,\delta \rightarrow
0^{+}}\int_{f^{-1}([c_{j}+\epsilon ,c_{j+1}-\delta ])}g\,d\mu \\
&=&\lim_{\epsilon ,\delta \rightarrow 0^{+}}\int_{c_{j}+\epsilon
}^{c_{j+1}-\delta }\left( \int_{f^{-1}(t)}\frac{g}{\left\vert \nabla
f\right\vert }d\sigma \right) dt \\
&=&\int_{c_{j}}^{c_{j+1}}\left( \int_{f^{-1}(t)}\frac{g}{\left\vert \nabla
f\right\vert }d\sigma \right) dt\,\text{.}
\end{eqnarray*}%
Summing these integrals over $j$ proves the assertion.
\end{proof}

Recall from \S 1 the mean curvature $H$ and Gaussian curvature $K$, both
regarded as functions on the set of regular points of $f$. Explicit formulae
are known for $H$ and $K$. To state them, let $Q$ be the Hessian quadratic
form associated with $f$ and define the quadratic form $Q^{\ast }$ to be the
one whose standard matrix is the adjugate (or \textquotedblleft classical
adjoint\textquotedblright ) of the standard matrix for $Q$; we shall regard
the two quadratic forms $Q$ and $Q^{\ast }$ as real-valued functions of one
vector variable. Then,%
\begin{equation*}
H=\frac{\left\vert \nabla f\right\vert ^{2}\limfunc{Tr}Q-Q(\nabla f)}{%
n\left\vert \nabla f\right\vert ^{3}}\text{\quad and\quad }K=\frac{Q^{\ast
}(\nabla f)}{\left\vert \nabla f\right\vert ^{n+2}}\,\text{.}
\end{equation*}%
These are implicit in \cite[p.\thinspace 204]{Spivak} and made explicit in 
\cite{Goldman}. (In both of these references, $f^{-1}(t)$ is oriented by $%
\nabla f/\left\vert \nabla f\right\vert $, the opposite of our choice of $%
\mathbf{N}$.)

\begin{lemma}
For a $C^{2}$ Morse function $f$ on an open set $\Omega \subset 
\mathbb{R}
^{n+1}$, the functions $H$, $K\partial _{i}f$, and $K\left\vert \nabla
f\right\vert $ are all integrable on any compact subset of $\Omega $.
\end{lemma}

\begin{proof}
It suffices to show that they are integrable \textquotedblleft
near\textquotedblright\ each critical point $p$, i.e., on a closed ball $D$
centered at $p$ in which $p$ is the only critical point. Without loss of
generality, assume that $p$ is the origin $\mathbf{0}\in 
\mathbb{R}
^{n+1}$. We notate a typical point in $%
\mathbb{R}
^{n+1}$ by writing its position vector $\mathbf{r}$ and we let $r=\left\Vert 
\mathbf{r}\right\Vert $. Then, for $\mathbf{r}$ near $\mathbf{0}$,%
\begin{equation*}
f(\mathbf{r})=f(\mathbf{0})+P(\mathbf{r})+o(r^{2})
\end{equation*}%
where $P(\mathbf{r})$ is the quadratic polynomial $\frac{1}{2}Q(\mathbf{r})$%
. For each $i\in \{1,\cdots ,n+1\}$,%
\begin{equation*}
\partial _{i}f=\partial _{i}P+r\epsilon _{i}\,\text{,}
\end{equation*}%
where $\epsilon _{i}\rightarrow 0$ as $r\rightarrow 0$, and for $\mathbf{r}%
\in D^{\prime }:=D\smallsetminus \{\mathbf{0}\}$,%
\begin{equation*}
\partial _{i}P(\mathbf{r})=r\alpha _{i}(\mathbf{r}/r)
\end{equation*}%
where $\alpha _{i}$ is a function on $S^{n}$. Hence, on $D^{\prime }$,%
\begin{equation*}
\left\vert \nabla f(\mathbf{r})\right\vert ^{2}=r^{2}\sum_{i=1}^{n+1}\left(
\alpha _{i}(\mathbf{r}/r)+\epsilon _{i}\right) ^{2}\,\text{.}
\end{equation*}%
As $f$ is a Morse function, $\mathbf{0}$ is the only critical point of $P$
and thus $\sum_{i}\alpha _{i}(\mathbf{r})^{2}>0$ for $\mathbf{r}\in S^{n}$.
Letting%
\begin{equation*}
m=\min_{\mathbf{r}\in S^{n}}\sum_{i=1}^{n+1}\alpha _{i}(\mathbf{r})^{2}\text{%
\thinspace ,}
\end{equation*}%
we have, for sufficiently small $r$, $\frac{1}{2}mr^{2}\leq \left\vert
\nabla f(\mathbf{r})\right\vert ^{2}\leq 2mr^{2}$. Hence, there are positive
numbers $C$,$\,M_{1}$,$\,M_{2}$,\thinspace $\delta $ such that, whenever $%
r\leq \delta $,%
\begin{equation*}
\left\vert \nabla f(\mathbf{r})\right\vert \geq Cr
\end{equation*}%
as well as%
\begin{equation*}
\left\vert \left\vert \nabla f\right\vert ^{2}\limfunc{Tr}Q-Q(\nabla
f)\right\vert (\mathbf{r})\leq M_{1}r^{2}\text{\quad and\quad }\left\vert
Q^{\ast }(\nabla f)\right\vert (\mathbf{r})\leq M_{2}r^{2}\text{\thinspace .}
\end{equation*}%
Therefore, for $r\leq \delta $,%
\begin{equation*}
\left\vert H(\mathbf{r})\right\vert =\frac{\left\vert \left\vert \nabla
f\right\vert ^{2}\limfunc{Tr}Q-Q(\nabla f)\right\vert (\mathbf{r})}{%
n\left\vert \nabla f(\mathbf{r})\right\vert ^{3}}\leq \frac{M_{1}}{nC^{3}}%
\frac{1}{r}
\end{equation*}%
and%
\begin{equation*}
\left\vert K(\mathbf{r})\partial _{i}f(\mathbf{r})\right\vert \leq
\left\vert K(\mathbf{r})\nabla f(\mathbf{r})\right\vert =\frac{\left\vert
Q^{\ast }(\nabla f)\right\vert (\mathbf{r})}{\left\vert \nabla f\right\vert
^{n+1}}\leq \frac{M_{2}}{C^{n+1}}\frac{1}{r^{n-1}}\text{\thinspace .}
\end{equation*}%
It is a standard fact that, for any $c>0$, $1/r^{n+1-c}$ is integrable on
any origin-centered ball in $%
\mathbb{R}
^{n+1}$. Hence, $H$, $K\partial _{i}f$, and $K\left\vert \nabla f\right\vert 
$ are all integrable on $D$.
\end{proof}

\section{Main Results}

We establish the main results of the article.

\begin{theorem}
\label{Theorem Mean Curvature}Under Hypothesis \dag , $\nu (b)-\nu
(a)=n\int_{f^{-1}([a,b])}H\,d\mu $.
\end{theorem}

\begin{proof}
First recall (from \cite[p.\thinspace 142]{do Carmo-2}) that $H=-\frac{1}{n}%
\func{div}\mathbf{N}$. With $\mathbf{N:=-}\nabla f/\left\vert \nabla
f\right\vert $,%
\begin{equation*}
H=\frac{1}{n}\func{div}\frac{\nabla f}{\left\vert \nabla f\right\vert }\text{%
\thinspace .}
\end{equation*}

In the following, let $R=f^{-1}([a,b])$. There are two cases according as
whether $[a,b]$ contains a critical value.

Case 1: $[a,b]$ is free of critical values. Then, $R$ is an $(n+1)$-manifold
with boundary $f^{-1}(a)\cup f^{-1}(b)$. Let $\mathbf{n}$ denote the unit
outward normal (relative to $R$) on $\partial R$; then $\mathbf{n}=-\nabla
f/\left\vert \nabla f\right\vert $ on $f^{-1}(a)$ and $\mathbf{n}=\nabla
f/\left\vert \nabla f\right\vert $ on $f^{-1}(b)$. Now,%
\begin{equation*}
\nu (b)-\nu (a)=\int_{\partial R}\left\langle \frac{\nabla f}{\left\vert
\nabla f\right\vert },\mathbf{n}\right\rangle d\sigma =\int_{R}\func{div}%
\frac{\nabla f}{\left\vert \nabla f\right\vert }d\mu =\int_{R}nH\,d\mu \,%
\text{.}
\end{equation*}

Case 2: $[a,b]$ contains a critical value. Let $S$ be the (finite)\ set of
critical values in $[a,b]$. Then, $(a,b)\smallsetminus S$ is a disjoint
union of finitely many intervals $I_{j}=(c_{j},c_{j+1})$ of regular values.
As $R=\cup _{j}f^{-1}(I_{j})\cup f^{-1}\left( S\cup \{a,b\}\right) $ and $%
f^{-1}\left( S\cup \{a,b\}\right) $ has Lebesgue measure zero,%
\begin{equation*}
\int_{R}H\,d\mu =\sum_{j}\int_{f^{-1}(I_{j})}H\,d\mu \,\text{.}
\end{equation*}%
It remains to note that, for each $j$,%
\begin{eqnarray*}
\int_{f^{-1}(I_{j})}H\,d\mu &=&\lim_{\epsilon \rightarrow
0^{+}}\int_{f^{-1}([c_{j}+\epsilon ,c_{j+1}-\epsilon ])}H\,d\mu \quad \text{%
(by integrability of }H\text{)} \\
&=&\lim_{\epsilon \rightarrow 0^{+}}\frac{1}{n}\left( \nu (c_{j+1}-\epsilon
)-\nu (c_{j}+\epsilon )\right) \quad \text{(by Case 1)} \\
&=&\frac{1}{n}\left( \nu (c_{j+1})-\nu (c_{j})\right) \,\quad \text{(by
continuity of }\nu \text{).}
\end{eqnarray*}
\end{proof}

With the aid of Lemma \ref{Lemma Key}, Theorem \ref{Theorem Mean Curvature}
easily yields a formula for $\nu ^{\prime }$, which would take considerable
effort to obtain otherwise.

\begin{corollary}
\label{Corollary v'}Assume Hypothesis \dag . For any regular value $t_{0}\in
\lbrack a,b]$,%
\begin{equation*}
\nu ^{\prime }(t_{0})=n\int_{f^{-1}(t_{0})}\frac{H}{\left\vert \nabla
f\right\vert }d\sigma \,\text{.}
\end{equation*}
\end{corollary}

\begin{proof}
For a regular value $t_{0}\in (a,b)$,%
\begin{eqnarray*}
\nu ^{\prime }(t_{0}) &=&\left. \frac{d}{dt}\right\vert
_{t_{0}}\int_{f^{-1}([a,t])}nH\,d\mu \quad \text{(by Theorem \ref{Theorem
Mean Curvature})} \\
&=&\left. \frac{d}{dt}\right\vert _{t_{0}}\int_{a}^{t}\left(
\int_{f^{-1}(\tau )}\frac{nH}{\left\vert \nabla f\right\vert }d\sigma
\right) d\tau \quad \text{(by Lemma \ref{Lemma Key})} \\
&=&\int_{f^{-1}(t_{0})}\frac{nH}{\left\vert \nabla f\right\vert }d\sigma
\quad \text{(by fundamental theorem of calculus).}
\end{eqnarray*}
\end{proof}

We show more applications of Lemma \ref{Lemma Key} with a certain choice of $%
g$.

\begin{theorem}
\label{Theorem Integral of Gradient}Under Hypothesis \dag , $%
\int_{f^{-1}[a,b]}(h\circ f)\cdot \left\vert \nabla f\right\vert d\mu
=\int_{a}^{b}h(t)\nu (t)\,dt$ for any integrable function $h$ on $[a,b]$. In
particular, for any $t_{0}\in \lbrack a,b]$,%
\begin{equation*}
\int_{a}^{t_{0}}\nu (t)\,dt=\int_{f^{-1}([a,t_{0}])}\left\vert \nabla
f\right\vert d\mu \,\text{,}
\end{equation*}%
or equivalently,%
\begin{equation*}
\nu (t_{0})=\left. \frac{d}{dt}\right\vert
_{t_{0}}\int_{f^{-1}([a,t])}\left\vert \nabla f\right\vert d\mu \,\text{.}
\end{equation*}
\end{theorem}

\begin{proof}
The first assertion follows from Lemma \ref{Lemma Key} by letting $g=(h\circ
f)\cdot \left\vert \nabla f\right\vert $. The second assertion results from
letting $h$ be the indicator function for $[a,t_{0}]$. Continuity of $\nu $
makes applicable the fundamental theorem of calculus, yielding the last
assertion.
\end{proof}

\begin{proposition}
Assume Hypothesis \dag .

\begin{description}
\item[(a)] $\int_{f^{-1}([a,b])}K\partial _{i}f\,d\mu =0$ for $i\in
\{1,\cdots ,n+1\}$.

\item[(b)] If, in addition, $n$ is even and $[a,b]$ is free of critical
values, then%
\begin{equation*}
\int_{f^{-1}([a,b])}K\left\vert \nabla f\right\vert d\mu =\frac{1}{2}%
(b-a)\chi (f^{-1}(a))\nu (S^{n})\text{\thinspace ,}
\end{equation*}%
where $\nu (S^{n})$ is the ($n$-dimensional) volume of the unit sphere $%
S^{n} $ and $\chi (f^{-1}(a))$ is the Euler characteristic of $f^{-1}(a)$.
\end{description}
\end{proposition}

\begin{proof}
For Part (a), let $g$ in Lemma \ref{Lemma Key} be the vector-valued function 
$K\nabla f$. Then,%
\begin{equation*}
\int_{f^{-1}([a,b])}K\nabla f\,d\mu =\int_{a}^{b}\left( \int_{f^{-1}(t)}K%
\frac{\nabla f}{\left\vert \nabla f\right\vert }d\sigma \right) dt\,\text{.}
\end{equation*}%
Now, note that%
\begin{equation*}
\int_{f^{-1}(t)}K\frac{\nabla f}{\left\vert \nabla f\right\vert }d\sigma
=-\int_{f^{-1}(t)}K\mathbf{N\,}d\sigma =\mathbf{0}\text{\thinspace .}
\end{equation*}%
For detail of the last equality, let $M$ denote $f^{-1}(t)$ and define the
vector-valued $n$-form $\omega $ on $S^{n}$ by letting $\omega =\limfunc{Id}%
_{S^{n}}d\sigma _{S^{n}}$, where $d\sigma _{S^{n}}$ is the volume form on $%
S^{n}$. Then, with $G$ being the Gauss map $p\mapsto \mathbf{N}(p)$, $%
G^{\ast }\omega =K\mathbf{N}\,d\sigma $ as can be verified pointwise. Hence,%
\begin{equation*}
\int_{M}K\mathbf{N}\,d\sigma =\int_{M}G^{\ast }\omega =\deg G\cdot
\int_{S^{n}}\omega \,\text{.}
\end{equation*}%
But%
\begin{equation*}
\int_{S^{n}}\omega =\int_{S^{n}}\limfunc{Id}\nolimits_{S^{n}}d\sigma
_{S^{n}}=\mathbf{0}
\end{equation*}%
due to cancellation of antipodal contributions.

Under the hypothesis of Part (b), $f^{-1}(t)$ is diffeomorphic to $f^{-1}(a)$
for $t\in \lbrack a,b]$. By Gauss-Bonnet theorem, 
\begin{equation*}
\int_{f^{-1}(t)}K\,d\sigma =\frac{1}{2}\chi (f^{-1}(t))\nu (S^{n})=\frac{1}{2%
}\chi (f^{-1}(a))\nu (S^{n})
\end{equation*}%
Letting $g=K\left\vert \nabla f\right\vert $ in Lemma \ref{Lemma Key} then
proves Part (b).
\end{proof}

\bigskip 

\noindent \textit{\small Mathematics Department, Illinois State University,
Normal, Illinois}

\noindent \texttt{pding@ilstu.edu}

\end{document}